\documentclass[12pt]{amsart}

\newtheorem{theorem}{Theorem}[section]

\newtheorem{lemma}[theorem]{Lemma}

\newtheorem{corollary}[theorem]{Corollary}
\newtheorem{conjecture}[theorem]{Conjecture}

\newcommand{\bF}{\mathbb F}

\newcommand{\bN}{\mathbb N}
\newcommand{\bZ}{\mathbb Z}
\newcommand{\cS}{\mathcal{S}}

\newcommand{\cP}{\mathcal{P}}

\newcommand{\cL}{\mathcal{L}}
\newcommand{\cR}{\mathcal{R}}
\newcommand{\cF}{\mathcal{F}}

\newcommand{\del}{\setminus}
\newcommand{\con}{/}

\DeclareMathOperator{\tower}{tower}
\DeclareMathOperator{\si}{si}

\DeclareMathOperator{\cl}{cl}
\DeclareMathOperator{\PG}{PG}

\DeclareMathOperator{\Span}{span}

\begin{document}

\sloppy

\title[On Rota's Conjecture]
{On Rota's conjecture and nested separations in matroids}

\author{Shalev Ben-David}
\address{Computer Science and Artificial Intelligence Laboratory,
MIT, Cambridge, Massachusetts, USA.}
\author{Jim Geelen}
\address{Department of Combinatorics and Optimization,
University of Waterloo, Waterloo, Canada}
% \email{jfgeelen@math.uwaterloo.ca}
\thanks{ This research was partially supported by an
NSERC Discovery Grant held by Jim Geelen and by an
NSERC Undergraduate Student Research Award held by
Shalev Ben-David at the University of Waterloo in 2009}

\date{\today}

\begin{abstract}
We prove that for each finite field $\bF$ and integer $k\in \bZ$
there exists $n\in \bZ$ such that no excluded minor for the 
class of $\bF$-representable matroids has $n$ nested 
$k$-separations.
\end{abstract}

\maketitle

\section{Introduction}
We prove a partial result towards
Rota's Conjecture~[\ref{rota}].
\begin{conjecture}[Rota]
For each finite field $\bF$,
there are, up to isomorphism, 
only finitely many excluded minors for
the class of $\bF$-representable matroids.
\end{conjecture}

A sequence $(A_1,B_1),\ldots,(A_n,B_n)$ of $k$-separations
in a matroid is said to be {\em nested} if
$A_1\subset A_2\subset \cdots\subset A_n$.
We prove the following theorem.
\begin{theorem}\label{main}
Let $\bF$ be a finite field of order $q$,
let $k$ be a positive integer, and let 
$n=\tower(q,q,k,6)$. Then
no excluded minor for the class of $\bF$-representable matroids
admits a sequence of $n$ nested $k$-separations.
\end{theorem}

Here
$\tower(a_1,a_2,\ldots,a_n)=
a_1^{ a_2^{\cdot^{\cdot^{\cdot^{a_n}}}}}.$

The special case of this result with $k=3$ was proved by 
Oxley, Vertigan, and Whittle (personal communication).

We conclude the introduction with an application
of Theorem~\ref{main} to branch-width;
this corollary is proved in Section~\ref{sec:bw}
where we also define branch-width.
\begin{corollary}\label{bw}
For any finite field $\bF$  of order $q$ and positive integer $k$,
if $M$ is an excluded minor for
the class of $\bF$-representable matroids and $M$ has branch-width
$k$, then $|M|\le \tower(3,q,q,k,6)$.
\end{corollary}
Corollary~\ref{bw} improves the main result in~[\ref{gw}] 
which gives a non-computable bound on $|M|$.

Our main result, Theorem~\ref{main2}, is an extension
of Theorem~\ref{main} that involves representabilty over several
fields. 

\section{preliminaries}

We use the following standard notation:
we denote the power set of a set $E$ by $2^E$, and,
if $f$ is a function whose domain is a set $E$ 
and $X\subseteq E$, then we denote
$\{f(x)\, : \, x\in X\}$ by $f(X)$.

We follow the terminology of Oxley~[\ref{oxley}],
except we write $|M|$ for the size of a matroid $M$.
For a finite field $\bF$, we define an {\em represented matroid}
to be a pair $(M,S)$ where $\si(S)$ is a projective geometry
over $\bF$ and $M$ is a restriction of $S$. 
For a represented matroid $(M,S)$ and $X\subseteq E(M)$,
we denote $\cl_S(X)$ by $\Span(X)$.
For disjoint sets
$D,C\subseteq E(M)$, we call $(M\del D\con C, S\con C)$
a {\em minor} of $(M,S)$.  For notational
convenience we will usually refer to the represented matroid
by $M$ alone, and write $S_M$ for $S$.

Let $M$ be a matroid.
For $X,Y\subseteq E(M)$, we define
\begin{eqnarray*}
\sqcap_M(X,Y)&=& r_M(X) + r_M(Y) - r_M(X\cup Y) \mbox{ and}\\
\lambda_M(X) &=& \sqcap_M(X,E(M)-X).
\end{eqnarray*}
Thus, if $M$ is representable, then
$$ \sqcap_M(X,Y) = r_{S_M}(\Span(X)\cap\Span(Y)).$$

It is well-known that $\lambda$ is submodular; that is,
$$ \lambda_M(X) + \lambda_M(Y) \ge 
\lambda_M(X\cap Y) + \lambda_M(X\cup Y).$$
If $X$ and $Y$ are disjoint, then  expanding the definitions
reveals that
\begin{equation*}
 \lambda_{M\con Y}(X) = \lambda_M(X) - \sqcap_M(X,Y).
\end{equation*}

A {\em $k$-separation}
is a partition $(A,B)$ of $M$ such that $|A|,|B|\ge k$ and
$\lambda_M(A)<k$; if $\lambda_M(A)=k-1$, then $(A,B)$
is said to be an {\em exact $k$-separation}.

For disjoint sets $X$ and $Y$ of elements, we define
$$\kappa_M(X,Y) = \min(\lambda_M(Z)\, : \, X\subseteq Z\subseteq E(M)-Y).$$
For any set $C\subseteq E(M)-(X\cup Y)$ it is straightforward
to show that $\sqcap_{M\con C}(X,Y)\le \kappa_M(X,Y)$.
The following result, due to Tutte (see~[\ref{oxley}, Theorem 8.5.2]),
shows that there exists $C$ for which equality is attained; this
differs from Tutte's formulation, but the two versions
are readily seen to be equivalent.
\begin{theorem}[Tutte's Linking Theorem]
If $X$ and $Y$ are  disjoint sets of elements in a matroid $M$,
then there is an independent set $C\subseteq E(M)-(X\cup Y)$
such that $\sqcap_{M\con C}(X,Y) = \kappa_M(X,Y)$.
\end{theorem}

We also require the following results.
\begin{lemma}\label{seq}
Let $(A_1,A_2,A_3,A_4)$ be a partition of the ground set of a matroid
$M$ and let $k\in\bN$ such that 
$\lambda_M(A_1\cup A_2)
=\kappa_M(A_1,A_3\cup A_4) = \kappa_M(A_1\cup A_2,A_4) = k$,
then $\kappa_M(A_1,A_4) = k$.
\end{lemma}

\begin{proof}
Let $X$ be a set with
$A_1\subseteq X \subseteq E(M)-A_4$ and
$\lambda_M(X) = \kappa_M(A_1,A_4)$. By submodularity,
$$ \lambda_M(X) + \lambda_M(A_1\cup A_2) \ge 
\lambda_M( A_1\cup (X\cap A_2) ) + 
\lambda_M(A_1\cup A_2\cup X).$$
Now $A_1\subseteq A_1\cup (X\cap A_2) \subseteq E(M) - (A_3\cup A_4)$ and
$A_1\cup A_2 \subseteq A_1\cup A_2\cup X \subseteq E(M) -  A_4$, so
$\lambda_M( A_1\cup (X\cap A_2) ) \ge \kappa_M(A_1,A_3\cup A_4) =k$ and
$\lambda_M( A_1\cup A_2\cup X ) \ge \kappa_M(A_1\cup A_2, A_4) =k$.
Moreover $\lambda_M(A_1\cup A_2) = k$, so
$\kappa_M(A_1,A_4) = \lambda_M(X) \ge k$.
On the other hand $\kappa_M(A_1,A_4)\le \lambda_M(A_1\cup A_2) = k$.
\end{proof}

\begin{lemma}\label{seqcon}
Let $M$ be a matroid, let $k\in \bN$, and let $S,T,X, C,D\subseteq 
E(M)$ such that
$(S,T,C,D)$ is a partition of $E(M)$,
$S\subseteq X\subset E(M)-T$, and
$k=\lambda_M(X) = \lambda_{M\con C\del D}(S)$.
Then $\lambda_{M\con (C-X)\del D-X}(X) = k$.
\end{lemma}

\begin{proof}
We have 
$k = \lambda_M(X) \ge
\lambda_{M\con (C-X)\del D-X}(X) ge
\lambda_{M\con C\del D}(S) \ge k$.
\end{proof}

The following lemma provides a crude upper bound on the number of flats
in a projective geometry.
\begin{lemma}\label{flats}
The number of flats of $PG(k-1,q)$ is at most $\tower(q,k,2)$.
\end{lemma}

\begin{proof} 
For each flat of $PG(k-1,q)$ there is a list of $k$ points
in GF$(q)^k$ that span the flat.
So the number of flats is at most $\left(q^k\right)^k=\tower(q,k,2)$.
\end{proof}

\section{Inequivalent representations and $k$-separations}

The fact that a matroid can have many inequivalent
representations over a field is a major cause of
difficulties in attacking Rota's Conjecture.
In this section we develop techniques that enable us to
control inequivalent representations relative to a 
given $k$-separation.

The following  two lemmas are primarily intended as motivation
for the definition of a ``scheme" that comes at the end of this section.
\begin{lemma}\label{motivation1}
Let $\bF$ be a finite field, let $M$ and $N$
be $\bF$-represented matroids with $E(M)=E(N)$,
and let $(A,B)$ be a partition of $E(M)$.
Then $M=N$ if and only if
\begin{itemize}
\item[(i)] $r_M(X) = r_N(X)$ for each $X\subseteq A$,
\item[(ii)] $r_M(Y) = r_N(Y)$ for each $Y\subseteq B$, and
\item[(iii)] $\sqcap_M(X,Y) = \sqcap_N(X,Y)$ for each
$X\subseteq A$ and $Y\subseteq B$.
\end{itemize}
\end{lemma}

\begin{proof} Conditions $(i)$, $(ii)$, and $(iii)$ are clearly 
nesessary, for the converse, suppose that they hold.
Then, for any set $Z\subseteq E(M)$, we have
\begin{eqnarray*}
 r_M(Z) &=& r_M(Z\cap A) + r_M(Z\cap B) - \sqcap_M(Z\cap A,Z\cap B) \\
  &=& r_N(Z\cap A) + r_N(Z\cap B) - \sqcap_N(Z\cap A,Z\cap B) \\
&=& r_N(Z).
\end{eqnarray*}
Thus $M=N$.
\end{proof}

\begin{lemma}\label{motivation2}
Let $\bF$ be a finite field, let $(A,B)$ be a
$k$-separation in a matroid $M$, let $X\subseteq A$
and $Y\subseteq B$, and let $W=\Span(A)\cap \Span(B)$.
Then $\sqcap_M(X,Y) = \sqcap_{S_M}(W\cap \Span(X), W\cap\Span(Y))$.
\end{lemma}

\begin{proof}
The result follows from the fact that 
$\Span(X)\cap\Span(Y) \subseteq \Span(A)\cap \Span(B) = W$.
\end{proof}

Let $(A,B)$ be a $k$-separation in a matroid $M$.
Two sets $X,Y\subseteq A$ are said to be {\em equivalent}
if $\sqcap_M(X,Z)=\sqcap_M(Y,Z)$ for all $Z\subseteq B$.
We let $\cP(M,A)$ denote the partition of the power set of
$A$ into equivalence classes.
\begin{lemma}\label{classes1}
Let $\bF$ be a finite field of order $q$, let $M$ be a 
$\bF$-represented matroid, let $k\ge 2$ be an integer,
and let $(A,B)$ be a $k$-separation in $M$.
Then $|\cP(M,A)|\le \tower(q,k-1,2)$.
\end{lemma}

\begin{proof}
By possibly reducing $k$,
we may assume that $(A,B)$ is an exact $k$-separation.
Let $W = \Span(X)\cap \Span(Y)$, thus
$\si(S_M|W)$ is isomorphic to $\PG(k-2,q)$.
Note that, for each flat $F$ of $S_M|W$, if
$X,Y\subseteq A$ such that $\Span(X)\cap W =F$
and $\Span(Y)\cap W =F$, then, by Lemma~\ref{motivation2},
the sets $X$ and $Y$ are equivalent. 
So the result follows by Lemma~\ref{flats}.
\end{proof}

The next result enables us to bound $|\cP(M,A)|$ when 
$M$ is an excluded minor.
\begin{lemma}\label{classes2}
Let $k\ge 2$ be an integer, let $(A,B)$ be a $k$-separation
in a matroid $M$, and let $e\in B$.  
If either $e\notin\cl_M(A)$ or $e\not\in\cl_{M^*}(A)$, then for each
equivalence class $P\in\cP(M,A)$ there exist equivalence classes
$P_1\in\cP(M\del e,A)$ and $P_2\in\cP(M\con e,A)$
such that $P=P_1\cap P_2$.
\end{lemma}

\begin{proof}
There exist equivalence classes 
$P_1\in\cP(M\del e,A)$ and $P_2\in\cP(M\con e,A)$
with $P\subseteq P_1$ and $P\subseteq P_2$.
Suppose that $P\neq P_1\cap P_2$.
Then there are inequivalent sets $X\in P$ and
$Y\in (P_1\cap P_2)-P$. So there is a set $Z\subseteq B$
such that $\sqcap_M(X,Z)\neq\sqcap_M(Y,Z)$.
If $e\not \in Z$, then $\sqcap_{M\del e}(X,Z)\neq\sqcap_{M\del e}(Y,Z)$,
contradicting that $X,Y\in P_1$.
Thus $e\in Z$. 
Suppose that $\sqcap_M(X,\{e\}) = \sqcap_M(Y,\{e\})$.
Then
\begin{eqnarray*}
\sqcap_{M\con e}(X,Z-\{e\})&=&
 \sqcap_{M}(X,Z) - \sqcap_M(X,\{e\}) \\
& \neq&  \sqcap_{M}(Y,Z)- \sqcap_M(Y,\{e\}) \\
&=& \sqcap_{M\con e}(Y,Z-\{e\}),
\end{eqnarray*}
contradicting that $X,Y\in P_2$.
Thus $\sqcap_M(X,\{e\}) \neq \sqcap_M(Y,\{e\})$. 
It follows that $e\in \cl_M(X\cup Y) \subseteq \cl_M(A)$.
By the hypotheses of the lemma, $e\not\in \cl_{M^*}(A)$.
Then, since $X,Y\in P_1$,
$\sqcap_M(X,B) = \sqcap_{M\del e}(X,B-\{e\})
=\sqcap_{M\del e}(Y,B-\{e\}) = \sqcap_M(Y,B)$.
It follows that
$\sqcap_{M\con e}(X,B-\{e\})
= \sqcap_M(X,B) - \sqcap_M(X,\{e\})
\neq \sqcap_M(Y,B) - \sqcap_M(Y,\{e\})
= \sqcap_{M\con e}(Y,B-\{e\}). $
This contradicts the fact that $X,Y\in P_2$.
\end{proof}

Let $(A,B)$ be a $k$-separation in a represenatable matroid
$M$ and let $W= \Span(A)\cap\Span(B)$.
The proof of Lemma~\ref{classes1} gives a sufficient
condition for two sets $X,Y\subseteq A$ to be equivalent;
namely that $\Span(X)\cap W = \Span(Y)\cap W$.
The following result characterizes the equivalent pairs;
this is a very tecnical result, but conceptually important.
The proof of the result follows directly from definitions,
so we omit it.
\begin{lemma}\label{equivalentpairs}
Let $\bF$ be a finite field of order $q$, let $M$ be a 
$\bF$-represented matroid, let $k\ge 2$ be an integer,
let $(A,B)$ be a $k$-separation in $M$, and let $W=\Span(A)\cap \Span(B)$.
Now, let $W_X = \Span(X)\cap W$ and $W_Y = \Span(Y)\cap W$,
and let $\cF$ denote the set of all flats $F$ of $S_M|W$ 
such that there exists $Z\subseteq B$ with
$\Span(Z)\cap W = F$.
Then $X$ and $Y$ are equivalent if and only if
$r_{S_M}(W_X\cap F) = r_{S_M}(W_Y\cap F)$ for each flat 
$F\in \cF$.
\end{lemma}

Let $(A,B)$ be an exact $k$-separation in a matroid $M$ and
let $\bF$ be a finite field of order $q$.
Let $\cF(k-2,\bF)$ be the set of all flats of $N=PG(k-2,\bF)$.
An {\em $\bF$-scheme}, or {\em scheme}, for $(A,B)$ is a function
$\sigma: \cP(M,A)\rightarrow 2^{\cF(k-2,\bF)}$.
A scheme $\sigma$ for $(A,B)$ is {\em realizable}
if there exists an $\bF$-representable matroid $M'$ such that
\begin{itemize}
\item $E(M') = A\cup E(N)$, $M'|A = M|A$, and $M'|E(N) = N$, and
\item for each $P\in \cP(M,A)$ and $X\in P$,
$\cl_{M'}(X)\cap E(N) \in \sigma(P)$.
\end{itemize}

We define a function $\pi: \cP(M,A)\times \cP(M,B)\rightarrow \bZ$
such that, for each $P_1\in \cP(M,A)$ and $P_2\in\cP(M,B)$,
$\pi(P_1,P_2) = \sqcap_M(X,Y)$ where $X\in P_1$ and $Y\in P_2$.
Let $\sigma_1$ be a scheme for $(A,B)$ and let 
$\sigma_2$ be a scheme for $(B,A)$. We say that
$\sigma_1$ and $\sigma_2$ are {\em compatible} if
for each $P_1\in\cP(M,A)$, $P_2\in\cP(M,B)$, $F_1\in\sigma_1(P_1)$,
and $F_2\in\sigma_2(P_2)$ we have
$\sqcap_N(F_1,F_2) = \pi(P_1,P_2)$.

The next lemma follows directly from these definitions.
\begin{lemma}\label{majic}
Let $(A,B)$ be an exact $k$-separation in a matroid $M$
and let $\bF$ be a finite field. Then $M$ is $\bF$-representable
if and only if there exist realizable schemes $\sigma_1$ 
for $(A,B)$ and $\sigma_2$ for $(B,A)$ that are compatible.
\end{lemma}

\begin{proof}
Suppose that $M$ is an $\bF$-represented matroid.
Let $N$ be a maximal simple restriction of
$S_M|(\Span(A)\cup \Span(B))$; thus 
$N$ is isomorphic to $\PG(k-2,\bF)$.
For each $P\in\cP(M,A)$ we let $\sigma_1(P)$ denote the
set of flats $\Span(X)\cap E(N)$ taken over all $X\in P$.
Similarly, for
each $P\in\cP(M,B)$ we let $\sigma_2(P)$ denote the
set of flats $\Span(X)\cap E(N)$ taken over all $X\in P$.
By definition, $\sigma_1$ is a realizable scheme for
$(A,B)$ and $\sigma_2$ is a realizable scheme for $(B,A)$.
By Lemmas~\ref{motivation1} and~\ref{motivation2},
$\sigma_1$ and $\sigma_2$ are compatible.

Conversely, suppose that
there exist realizable schemes $\sigma_1$ 
for $(A,B)$ and $\sigma_2$ for $(B,A)$ that are compatible.
Thus there are $\bF$-represented matroids 
$M_A$ and $M_B$ such that
\begin{itemize}
\item $E(M_A) = A\cup E(N)$, $M_A|A = M|A$, and $M_A|E(N) = N$, 
\item for each $P\in \cP(M,A)$ and $X\in P$,
$\cl_{M_A}(X)\cap E(N) \in \sigma_1(P)$,
\item $E(M_B) = B\cup E(N)$, $M_B|B = M|B$, and $M_B|E(N) = N$, 
\item for each $P\in \cP(M,B)$ and $X\in P$,
$\cl_{M_B}(X)\cap E(N) \in \sigma_2(P)$.
\end{itemize}
Now let $M'$ be the $\bF$-represented matroid
on ground set $E(M_A)\cup E(M_B)$ such that
$M'|{E(M_A)} = M_A$, $M'|{E(M_B)} = M_B$, and
$\Span(E(M_A))\cap \Span(E(M_B)) = \Span(E(N))$.
By the definition of compatible and
Lemma~\ref{motivation1},
$M'\del E(N)$ is a representation of $M$.
\end{proof}

\section{Nested $k$-separations}

A {\em $k$-dissection of length $t$} of a matroid $M$ is an
ordered partition $(A_0,A_1,\ldots,A_t)$
of $E(M)$ into nonempty sets such that,
for each $i\in \{1,\ldots,t\}$,
$(A_0\cup\cdots\cup A_{i-1},A_i\cup\cdots\cup A_t)$
is a $k+1$-separation. Thus, $M$ has a $k$-dissection of 
length $t$ if and only if $M$ has $t$ nested $k+1$-separations.
For brevity, we will write $A[i,j]$ for $A_i\cup\cdots\cup A_j$.
We say that a $k$-dissection $(X_0,\ldots,X_t)$ {\em contains}
a $k$-dissection $(Y_0,\ldots,Y_s)$ if 
there is an increasing function $f:\{0,\ldots,s\}
\rightarrow\{0,\ldots,t\}$ with $f(0)=0$ such that
for each $i\in\{0,\ldots,s-1\}$,
$Y_{i} = X[f(i),f(i+1)-1]$.
A $k$-dissection $(X_0,\ldots,X_t)$ is {\em linked} if
$\kappa_{M}(X_0,X_t) = k$.
Note that the set of linked $k$-dissections 
is closed under containment.
\begin{lemma}\label{linked}
If a matroid has a $k$-dissection of length $n^{k+1}$,
then it has a linked $l$-dissection of length $n$
for some $l\le k$.
\end{lemma}

\begin{proof}
Let $M$ be a matroid that admits a $k$-dissection
$(A_0,A_1,\ldots,A_t)$ with $t=n^{k+1}$.
Inductively we may assume that $M$ does not admit a
$(k-1)$-dissection of length $n^k$.
For each $i\in\{0,\ldots,n^k-1\}$, we may assume
that $\kappa_M(A[0,i n],A[(i+1)n,t])<k$ since otherwise
$(A[0,i n], A(i n+1), A(i n+2),\ldots, A[(i+1)n -1,A[(i+1)n,t])$ is a linked $k$-dissection of length $n$.
So there is a partition $(B_{i},C_{i+1})$ of
$A[i n+1, (i+1)n-1]$ such that $\lambda_M(A[0,i n]\cup B_i)<k$.
But then $(A_0\cup B_0, C_1\cup A_{n}\cup B_1,
 C_2\cup A_{2n}\cup B_2, \ldots, C_{n^k-1}\cup A_{(n^k-1)n}\cup B_{n^k-1}, C_{n^k}A_t)$ is a $k-1$-dissection of length $n^k$, contrary to 
our choice of $(A_0,A_1,\ldots,A_t)$.
\end{proof}

In the remainder of the paper, we will
be dealing with a linked $k$-dissection
$(A_0,\ldots,A_t)$ in a matroid $M$.
By Tutte's Linking Theorem we will find 
a partition $(C,D)$ of $A[1,\ldots,t-1]$
such that $\sqcap_{M\con C}(A_0,A_t) = k$.
For $X\subseteq C\cup D$ we will
denote $M\del (D\cap X)\con (C\cap X)$ by $M\circ X$.

We need the following technical result on dissections
of representable matroids.
\begin{lemma}\label{technical1}
Let $\bF$ be a finite field of order $q$ and
let $a,b,k\in \bN$ with $a \ge b \tower(q,k,2)^2$.
Now let $(A_0,A_1,\ldots,A_a)$
be a linked $k$-dissection of an $\bF$-represented
matroid $M$ and let $(C,D)$ be a partition of
$A[1,a-1]$ such that $\sqcap_{M\con C}(A_0,A_a) = k$.
Then  $(A_0,A_1,\ldots,A_a)$ contains a
$k$-dissection $(B_0,B_1,\ldots,B_b)$ 
such that, for each $1\le i< j\le b$, 
\begin{eqnarray*}
\cP(M\circ B[i,j-1], B[0,i-1])
&=&\cP(M, B[0,i-1]) \\
\cP(M\circ B[i,j-1], B[j,b])
&=&\cP(M, B[j,b]), \\
|\cP(M,B[0,i-1])| &=& |\cP(M,B[0,j-1])|,\mbox{ and} \\
|\cP(M,B[i,b])| &=& |\cP(M,B[j,b])|. 
\end{eqnarray*}
\end{lemma}

\begin{proof}
For each $i\in\{1,\ldots,a\}$ let 
$N_i$ be a maximal simple restriction of
$S_M| (\Span(A[0,i-1])\cap\Span(A[i,a]))$;
thus $N_i\cong \PG(k-1,\bF)$.
Consider a pair $(i,j)$ with $1\le i < j\le a$.
By Lemma~\ref{seqcon},
$\lambda_{M\circ A[i,j-1]}(A[0,i-1]) = k$.
Now consider the restriction $M_{ij}$ of $S_M$
to $E(M)\cup E(N_i)\cup E(N_j)$.
Now $N_i$ and $N_j$ are both restrictions of 
$M_{ij}\circ A[i,j-1]$ and each element of $N_i$ is
in parallel with a unique element of $N_j$;
let $\phi_{ij}$ denote the associated bijection.

For each $i\in\{1,\ldots,a\}$,
let $L_i$ denote the set
of flats $F$ of $N_i$ such that there exists
$X\subseteq A[0,i-1]$ with $\cl_N(X)\cap E(N_i) = F$ and let
$R_i$ denote the set
of flats $F$ of $N_i$ such that there exists
$X\subseteq A[i,a]$ with $\cl_N(X)\cap E(N_i) = F$.
By Lemma~\ref{flats}, $1\le |L_i|\le \tower(q,k,2)$
and $1\le |R_i|\le \tower(q,k,2)$.
So there exists $Z\subseteq \{1,\ldots,a\}$ with $|Z| = b$ and
integers $l,r\in \{1,\ldots,\tower(q,k,2)\}$ such that
$|L_i| = l$ and $|R_i|=r$ for each $i\in Z$.

Consider $i, j\in Z$ with $j>i$ and sets $X_1,X_2\subseteq A[0,i-1]$.
Let $F_1 = \cl_N(X_1)\cap E(N_i)$ and
$F_2 = \cl_N(X_2)\cap E(N_i)$. Note that
$F_1,F_2\in L_i$. Moreover $X_1$ and $X_2$ are not equivalent
if and only if there exists $F'\in R_i$ such that
$r_{S_M}(F'\cap F_1)\neq r_{S_M}(F'\cap F_2)$.
For each $R\in R_i$, we have $\phi_{ij}(R)\in R_j$;
so, since $|R_i|=|R_j|$,
$\phi_{ij}$ defines a bijection between $R_i$ and $R_j$.
So $X_1$ and $X_2$ are equivalent in $M$ if and only
if they are equivalent in $M\circ A[i,j-1]$.
Thus $\cP(M,A[0,i-1]) = \cP(M\circ A[i,j],A[0,i-1])$.
By symmetry,
$\cP(M,A[j,a]) = \cP(M\circ A[i,j-1],A[j,a])$.
Similarly, by Lemma~\ref{equivalentpairs}, 
\begin{eqnarray*}
|\cP(M,B[0,i-1])| &=& |\cP(M,B[0,j-1])|,\mbox{ and} \\
|\cP(M,B[i,b])| &=& |\cP(M,B[j,b])|. 
\end{eqnarray*}

Now the result holds by taking the $k$-dissection
corresponding to $Z$.
\end{proof}

\section{Excluded minors}

For a set $\cF$ of fields, we call a matroid 
{\em $\cF$-representable} 
if $M$ is $\bF$-representable for some $\bF\in \cF$.
Our main result is an extension of
Theorem~\ref{main} to $\cF$-representability.
To make the sums easier, it is convenient
to consider $2$-separations separately;
the following result is routine, we leave the proof to the reader.
\begin{lemma}\label{2seps}
Let $\cF$ be a finite set of fields.
Then each excluded minor for the class of
$\cF$-representable matroids has at most
$|\cF|-1$ nested $2$-separations.
\end{lemma}

The following technical result extends
Lemma~\ref{technical1} to excluded minors.
\begin{lemma}\label{technical2}
Let $\cF$ be a set of finite fields each of order $\le q$ and
let $a,n,k\in \bN$ with $a\ge 3$, $k\ge 2$, and
$n \ge a |\cF|\tower(q,k,6)$.
Now let $M$ be an excluded minor for the class of
$\cF$-representable matroids.
If $M$ has a linked $k$-dissection of length $n$, then
there exists
a linked $k$-dissection $(A_0,A_1,\ldots,A_a)$ of $M$,
integers $l,r\in \bN$,
a partition $(C,D)$ of $A[1,a-1]$, and
a field $\bF\in \cF$ such that
\begin{itemize}
\item
$\sqcap_{M\con C}(A_0,A_a) = k$;
\item
$M\circ A_i$ is $\bF$-representable for each $i\in \{1,\ldots,a-1\}$;
\item
for each $1\le i< j< a$, 
\begin{eqnarray*}
\cP(M\circ A[i,j], A[0,i-1]) &=&\cP(M, A[0,i-1]) \mbox{ and}\\
\cP(M\circ A[i,j], A[j+1,a]) &=&\cP(M, A[j+1,a]); \mbox{ and}
\end{eqnarray*}
\item for each $1\le i\le a$, 
\begin{eqnarray*}
|\cP(M, A[0,i-1])| &=& l \mbox{ and}\\
|\cP(M, A[i,a])| &=& r.
\end{eqnarray*}
\end{itemize}
\end{lemma}

\begin{proof}
Let $b=a\tower(q,k,2)^8$ and $c=|\cF|b$.
Since $k\ge 2$ and $q\ge 2$, it is a routine
calculation to check that
$n\ge c + 4q^k.$
So $M$ has  a linked $k$-dissection
$(C_0,\ldots,C_d)$ such that
$|C_0|\ge 2q^k$ and $|C_c|\ge 2q^k$.
By Tutte's Linking Theorem, there is a partition $(C,D)$ of
$C[1,c-1]$ such that $\sqcap_{M\con C}(C_0,C_c)=k$.
Since $M$ is an excluded-minor for the class of 
$\cF$-representable matroids, for each $i\in\{1,\ldots,c-1\}$
there exists a field $\bF_i\in \cF$ such that
$M\circ C_i$ is $\bF_i$-representable.
By majority, there exists a field
$\bF\in \cF$ and a linked $k$-dissection 
$(B_0,\ldots,B_b)$  of length $b$
contained in $(C_0,\ldots,C_c)$ such that
$M\circ B_i$ is $\bF$-representable for each $i\in\{1,\ldots,b-1\}$.

Note that $M$ is simple and cosimple.
So, for any $k+1$-separation $(X,Y)$ of 
$M$, we have $|X\cap\cl_M(Y)|\le q^k$ and
$|X\cap\cl^*(M)|\le q^k$.
Therefore there exists $e\in B_0$
that is in neither the closure nor the
coclosure of $B[1,b]$ and
there exists $f\in B_b$
that is in neither the closure nor the
coclosure of $B[0,b-1]$.
Now $(B_0-\{e\},B_1,\ldots,B_b)$ is a linked
$k$-dissection in both $M\del e$ and $M\con e$
and $(B_0,B_1,\ldots,B_b-\{f\})$ is a linked
$k$-dissection in both $M\del f$ and $M\con f$.
Moreover each of $M\del e$, $M\con e$, 
$M\del f$ and $M\con f$ is representable 
over some field in $\cF$.
We will now apply Lemma~\ref{technical1} to
each of the matroids $M\del e$, $M\con e$, 
$M\del f$, and $M\con f$ in turn to get increasingly coarser
$k$-dissections; let $(A_0,\ldots,A_a)$ be the final
$k$-dissection.
Then, by Lemma~\ref{classes2},
for each $1\le i< j< a$ we have 
\begin{eqnarray*}
\cP(M\circ A[i,j]), A[0,i-1]) &=&\cP(M, A[0,i-1]) \\
\cP(M\circ A[i,j]), A[j+1,a]) &=&\cP(M, A[j+1,a]) \\
|\cP(M,A[0,i-1])| &=& |\cP(M,A[0,j-1])|,\mbox{ and} \\
|\cP(M,A[i,a])| &=& |\cP(M,A[j,a])|. 
\end{eqnarray*}
Finally, let $l=|\cP(M,A_0)$ and $r=|\cP(M,A_a)$, and
replace $C$ and $D$ with $C\cap A[1,a-1]$ and $D\cap A[1,a-1]$.
\end{proof}

We can now prove our main result.
\begin{theorem}\label{main2}
Let $\cF$ be a finite set of finite fields each of size at most $q$,
let $k\in \bN$, 
and let $t= |\cF|^{k+1}\tower(q,q,k,6)$.
Then no excluded-minor for the class of $\cF$-representable matroids
admits a sequence of $t$ nested $k$-separations.
\end{theorem}

\begin{proof}
Let $M$ be an excluded-minor for the class of
$\cF$-representable matroids and suppose that
$M$ has $t$-nested $k$-separations.
By Lemma~\ref{2seps}, we may assume that $k\ge 2$.
Let $a = \tower(q,k,2)^{2\tower(q,k,2)} + 1$.
It is  a routine calculation to check that
$$ t \ge 
 \left( a|\cF| \tower(q,k,6) \right)^{k+1}.$$
Then, by Lemmas~\ref{linked} and~\ref{technical2},
there is a field $\bF\in\cF$, an integer $k'\le k$, a linked 
$k'$-dissection $(A_0,A_1,\ldots,A_a)$ of $M$,
integers $l,r\in\bN$,
and a partition $(D,C)$ of $E(M)-(A_0\cup A_a)$
such that:
\begin{itemize}
\item $\sqcap_{M\con C}(A_0,A_a) = k'$;
\item 
$M\circ A_i$ is $\bF$-representable for each $i\in \{1,\ldots,a-1\}$;
\item 
for each $1\le i< j< a$; 
\begin{eqnarray*}
\cP(M\circ A[i,j]), A[0,i-1]) &=&\cP(M, A[0,i-1]) \mbox{ and}\\
\cP(M\circ A[i,j]), A[j+1,a]) &=&\cP(M, A[j+1,a]); \mbox{ and}
\end{eqnarray*}
\item
for each $1\le i\le a$, 
\begin{eqnarray*}
|\cP(M, A[0,i-1])| &=& l \mbox{ and}\\
|\cP(M, A[i,b])| &=& r.
\end{eqnarray*}
\end{itemize}

Let $\cP(M,A_0) = \{\cL_1,\ldots,\cL_l\}$ 
and let $\cP(M,A_a)=\{\cR_1,\ldots,\cR_r\}$.
Now, for each $i\in\{0,\ldots,a-1\}$ 
and each $j\in \{1,\ldots,l\}$,
let $\cL_{ij}$ denote the equivalence class
of $\cP(M,A[0,i])$ that contains
$$ \{ L\cup (C\cap A[1,i])\, : \, L\in \cL_j \}.$$
Similarly, for each $i\in\{0,\ldots,a-1\}$ 
and each $j\in \{1,\ldots,r\}$,
let $\cR_{ij}$ denote the equivalence class
of $\cP(M,A[i+1,a])$ that contains
$$ \{ R\cup (C\cap A[i+1,a])\, : \, R\in \cR_j \}.$$

Recall that an $\bF$-scheme for
$(A[0,i],A[i+1,a])$ is a function
from $\cP(M,A[0,i])$ to $2^{\cF(k'-1,\bF)}$
and that $\cP(M,A[0,i]) = \{\cL_{i1},\ldots,\cL_{il}\}$.
Henceforth we will abuse notation by considering
an $\bF$-scheme  of $(A[0,i],A[i+1,a])$ as a function from
$\{1,\ldots,l\}$ to $2^{\cF(k'-1,\bF)}$.
Similarly, we will consider
an $\bF$-scheme  of $(A[i+1,a],A[0,i])$ as a function from
$\{1,\ldots,r\}$ to $2^{\cF(m-1,\bF)}$.
For each $i\in \{1,\ldots,a-1\}$,
let $\cS_L(i)$ denote the set of
all realizable $\bF$-schemes of
$(A[0,i],A[i+1,a])$ and 
let $\cS_R(i)$ denote the set of
all realizable $\bF$-schemes of
$(A[i+1,a],A[0,i])$. Since 
$M$ is not $\bF$-representable,
there are no compatible pairs
of $\bF$-schemes in $\cS_L(i)$ and
$\cS_R(i)$. 
By Lemma~\ref{classes1},
$l,r\le \tower(q,k,2)$. So, by Lemma~\ref{flats},
the number of pairs of functions
$(\sigma_L,\sigma_R)$ where
$\sigma_L:\{1,\ldots,l\}\rightarrow 2^{\cF(k'-1,\bF)}$
and
$\sigma_R:\{1,\ldots,r\}\rightarrow 2^{\cF(k'-1,\bF)}$
is at most
$$\tower(q,k,2)^{l+r} \le \tower(q,k,2)^{2\tower(q,k,2)}=a-1.$$
Therefore, there exists $0<i<j<a$ such that 
$\cS_L(i) = \cS_L(j)$ and $\cS_R(i)=\cS_R(j)$.
Then, by Lemma~\ref{majic}, $M\circ A[i,j-1]$
is not $\bF$-representable, which contradicts
our choice of $(A_0,\ldots,A_a)$.
\end{proof}

Note that Theorem~\ref{main} is an immediate corollary.

\section{Branch width}\label{sec:bw}

A tree is {\em cubic} if its internal vertices all have degree $3$.
The {\em leaves} of a tree are its degree-$1$ vertices.
Let $M$ be a matroid with $|M|\ge 2$.
A {\em branch-decomposition} of $M$ is a cubic tree $T$ whose
leaves are bijectively labelled by the elements of $M$.
If $T'$ is a subgraph of $T$
and $X\subseteq E(M)$ is the set of labels of $T'$,
then we say that $T'$ {\em displays} $X$.
The {\em width} of an edge $e$ of $T$ is defined to be
$\lambda_M(X)+1$ where $X$ is the set displayed by
one of the components of $T\del e$.  The {\em width} of $T$
is the maximum among the widths of its edges.
Finally, the {\em branch-width} of $M$ is the minimum among
the widths of all branch-decompositions of $M$.
For a matroid $M$ with $|M|\le 1$, the branch-width is defined
to be $|M|$.

The following result is a generalization of Corollary~\ref{bw}.
\begin{corollary}\label{bw2}
Let $\cF$ be a set of finite fields each of size at most $q$.
For each positive integer $k$,
if $M$ is an excluded minor for
the class of $\bF$-representable matroids and $M$ has branch-width
$k$, then $|M|\le \tower(3,q,q,k,6)$.
\end{corollary}

\begin{proof}
Consider a tree-decomposition $T$ of $M$ of 
width at most $k$. By Theorem~\ref{main2}, $M$ has no nested sequence 
of $k$-separations of length $\tower(q,q,k,q)$. Therefore $T$ has no 
Path of length $l=\tower(q,q,k,6) + 2k$.
The maximum number of leaves in a cubic tree with no 
path of length $l$ is
$$
= \left\{
\begin{array}{ll} 3\cdot 2^\frac{l-3}{2}, & l\mbox{ odd} \\
2\cdot 2^{\frac{l-2}{2}}, & l\mbox{ even}.
\end{array}
\right.
$$
In either case, the number of leaves is at most $\tower(3,q,q,k,6)$,
and hence $|M|\le \tower(3,q,q,k,6)$.
\end{proof}

\section*{Acknowledgement}
We thank the referees for their careful reading of this paper.

\section*{References}

\newcounter{refs}

\begin{list}{[\arabic{refs}]}%
{\usecounter{refs}\setlength{\leftmargin}{10mm}\setlength{\itemsep}{0mm}}

\item \label{gw}
J. Geelen, G. Whittle,  
Branch-width and Rota's Conjecture,
J. Combin. Theory Ser. B {\bf 86} (2002), 315-330.

\item \label{oxley}
J. G. Oxley,  {\em Matroid Theory}
Oxford University Press, New York, second edition (2011).

\item\label{rota}
G.-C. Rota,
Combinatorial theory, old and new.
In {\em Proc. Internat. Cong. Math.}
(Nice, Sept. 1970), pp. 229-233.
Gunthier-Villars, Paris.

\end{list}

\end{document}